\documentclass[11pt]{article}

\usepackage{amssymb}
\usepackage{amsmath}
\usepackage{graphicx}
\usepackage{color}
\usepackage{subcaption}
\usepackage[font={it}]{caption}

\usepackage[margin=1in]{geometry}
\allowdisplaybreaks

\usepackage{hyperref}

\usepackage{amsthm}
\newtheorem{thm}{Theorem}[section] 

\newtheorem{lem}[thm]{Lemma} 
\newtheorem{lemma}[thm]{Lemma} 
\newtheorem{claim}[thm]{Claim}

\newtheorem{prop}[thm]{Proposition}

\theoremstyle{definition}

\newcommand{\M}{\mathcal{M}}
\newcommand{\Mn}{\M_k}
\newcommand{\Mnb}{\M_k^{block}}

\newcommand{\bbR}{\mathbb{R}}
\newcommand{\mcE}{\mathcal{E}}

\newcommand{\Var}{\mathrm{var}}
\newcommand{\tmix}{t_{{\rm mix}}}
\newcommand{\trel}{t_{{\rm rel}}}

\title{Polynomial Mixing of the Edge-Flip Markov Chain for Unbiased Dyadic Tilings}
\date{}
\author{Sarah Cannon\footnote{College of Computing, Georgia Institute of Technology, Atlanta, GA 30332-0765, {\tt sarah.cannon@gatech.edu}. Supported in part by NSF DGE-1650044 and a grant from the Simons Foundation (\#361047 to Sarah Cannon).}, David Levin\footnote{Department of Mathematics, University of Oregon, Eugene, OR 97403-1222, {\tt dlevin@uoregon.edu}}, Alexandre Stauffer\footnote{Department of Mathematical Sciences, University of Bath, UK, {\tt a.stauffer@bath.ac.uk}. Supported by a Marie Curie Career Integration Grant PCIG13-GA-2013-618588 DSRELIS, and an EPSRC Early Career Fellowship.}}

\begin{document}
\maketitle
\thispagestyle{empty}
\begin{abstract}\normalsize
We give the first polynomial upper bound on the mixing time of the edge-flip Markov chain for unbiased dyadic tilings, resolving an open problem originally posed by Janson, Randall, and Spencer in 2002~\cite{jrs02}. 
A {\it dyadic tiling} of size $n$ is a tiling of the unit square by $n$ non-overlapping dyadic rectangles, each of area $1/n$, 
where a {\it dyadic rectangle} is any rectangle that can be written in the form $[a2^{-s}, (a+1)2^{-s}] \times [b2^{-t}, (b+1)2^{-t}]$ for $a,b,s,t \in \mathbb{Z}_{\geq 0}$. 
The edge-flip Markov chain selects a random edge of the tiling and replaces it with its perpendicular bisector if doing so yields a valid dyadic tiling.
Specifically, we show that the relaxation time of the edge-flip Markov chain for dyadic tilings is at most $O(n^{4.09})$, which implies that the mixing time is at most $O(n^{5.09})$. 
We complement this by showing that the relaxation time is at least $\Omega(n^{1.38})$, improving upon the previously best lower bound of $\Omega(n\log n)$ coming from the diameter of the chain. 
\end{abstract}


\section{Introduction}
We study the {edge-flip Markov chain} for dyadic tilings.  An interval is \emph{dyadic} if it can be written in the form $[a2^{-s}, (a+1)2^{-s}]$ for non-negative integers $a$ and $s$ with $0 \leq a < 2^{s}$. 
A rectangle is dyadic if it is the Cartesian product of two dyadic intervals. A \emph{dyadic tiling of size $n$} is a tiling of the unit square by 
$n$ non-overlapping dyadic rectangles with the same area $1/n$; see Figure~\ref{fig:dyadic_ex}. 
Less formally, work of Lagarias, Spencer, and Vinson \cite{lsv02} showed  that in two dimensions, dyadic tilings are precisely those tilings that can be constructed by bisecting the unit square, either horizontally or vertically; bisecting each half again, either horizontally or vertically; 
and repeatedly bisecting all remaining rectangular regions until there are $n$ total dyadic rectangles, each of equal area. We necessarily assume $n$ is a power of $2$. There is a natural Markov chain which connects the state space of all dyadic tilings of size $n$ by moves we refer to as {\it edge-flips}.

\begin{figure}
\centering
\includegraphics[scale = 0.65]{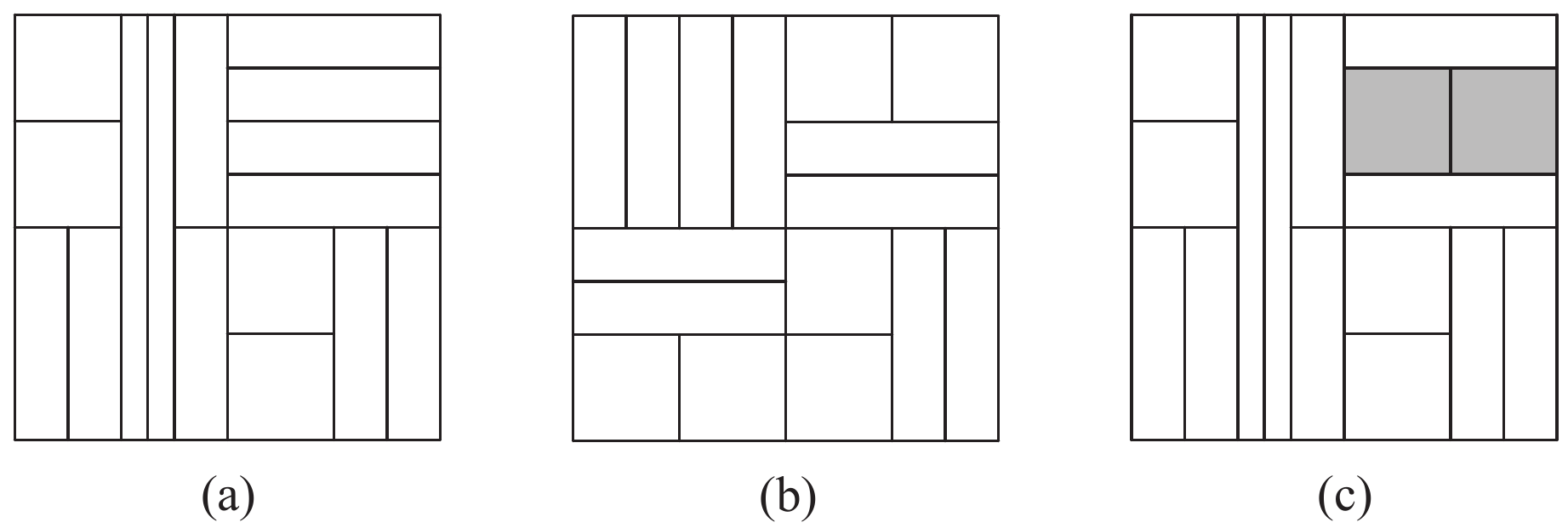} \vspace{-2mm}
\caption{ (a) A dyadic tiling of size 16 with a vertical bisector. (b) A dyadic tiling of size 16 with both a vertical and horizontal bisector. (c) A tiling that is not dyadic; the vertical component of the shaded rectangles is not a dyadic interval.}\label{fig:dyadic_ex}
\end{figure}

We analyze this \emph{edge-flip Markov chain} over the set of dyadic tilings of size $n$. 
Given any dyadic tiling, this chain evolves by selecting an edge of the tiling uniformly at random and replacing it by its perpendicular bisector, if doing so yields a valid dyadic tiling of size $n$;
an illustration is given in Figure~\ref{fig:edgeflip_ex}(a).
Our main result gives the first polynomial upper bound for the mixing time of this Markov chain. 
(The precise definitions of mixing time and relaxation time are deferred to Section~\ref{sec:markovchains}.)
\emph{In this paper, all logarithms have base 2.}
\begin{thm}\label{thm:upperbound}
   The relaxation time of the edge-flip Markov chain for dyadic tilings of size $n$ is at most $O(n^{\log 17})$. As a consequence, the mixing time of this chain is 
   at most $O(n^{1+\log 17})$.
\end{thm}

In terms of lower bounds, the best previously known lower bound for the mixing time is $\Omega(n\log n)$, which is a simple consequence of the fact that the diameter of the Markov chain is 
of order $n\log n$~\cite{jrs02}. In the theorem below we improve upon this bound, showing that even the relaxation time is much larger than $n\log n$.

\begin{thm}\label{thm:lowerbound}
   The relaxation time of the edge-flip Markov chain for dyadic tilings of size $n$ is at least $\Omega(n^{2\log \phi})$,
   where $\phi=\frac{\sqrt{5}+1}{2}$ is the golden ratio.
\end{thm}


\vspace{1mm}
\noindent
{\bf Related work.}
The edge-flip Markov chain for dyadic tilings was first considered by Janson, Randall, and Spencer in 2002~\cite{jrs02}. 
They showed that this Markov chain is irreducible, but left as an open problem to derive that the mixing time is polynomial in $n$. 
Instead, they presented another Markov chain, which has additional global moves consisting of rotations at all scales, 
and showed that this chain mixes in polynomial time. 
However, applications of the comparison technique of Diaconis and Saloff-Coste~\cite{dsc} have failed to extend this polynomial mixing bound to the 
more natural edge-flip Markov chain (which, in fact, corresponds to only performing rotations at the smallest scale).

Cannon, Miracle, and Randall considered the mixing time of the edge-flip Markov chain for a weighted version of dyadic tilings~\cite{cmr15}. In this version, given a parameter $\lambda>0$, 
the stationary probability of a dyadic tiling $x$ is proportional to $\lambda^{|x|}$, where $|x|$ is the sum of the length of the edges of $x$.
The Metropolis rule \cite{met} is incorporated into the edge-flip Markov chain so that the chain has the desired stationary distribution. 
They showed the mixing time of this chain is at least exponential in $n^2$ for any $\lambda>1$, and at most $O(n^2\log n)$ for any $\lambda<1$. 
This establishes a phase transition at critical point $\lambda = 1$, which corresponds to the unweighted case considered here. However, their techniques did not extend to the critical point, and they left as an open problem bounding the mixing time when $\lambda=1$.
Our main result, Theorem~\ref{thm:upperbound},  uses a different, non-local approach to
finally answer the question of~\cite{jrs02} and~\cite{cmr15} by showing the mixing time of the edge-flip Markov chain at critical point $\lambda = 1$ is at most polynomial in $n$, substantially less than the mixing time when $\lambda > 1$.
Furthermore, our Theorem
~\ref{thm:lowerbound} combined with the result for the weighted case in~\cite{cmr15} shows that the behavior at the (unweighted) critical point $\lambda=1$ 
is also substantially different than when $\lambda<1$. 
While it follows from the path coupling analysis in~\cite{cmr15} that the relaxation time is $O(n)$ for all fixed $\lambda<1$, 
 Theorem~\ref{thm:lowerbound} establishes a super-linear lower bound on the relaxation
time when~$\lambda=1$.

Variants of the edge-flip Markov chain offer a natural way to sample from many structures, 
but establishing rigorous polynomial upper bounds on the mixing time has often proven difficult, even in simple cases.
Perhaps the most studied case is that of triangulations of a given point set, 
as efficiently generating uniformly random triangulations of general planar point sets has been a problem of great interest in computer graphics and computational geometry. 
However, the mixing time of the edge-flip Markov chain for triangulations remains open in the general case, and no polynomial upper bound is known. 
The only known exception is for $n$ points in \emph{convex position}, which corresponds to triangulations of a convex polygon. 
In this case, the edge-flip Markov chain is known to mix in at most $O(n^5)$ steps~\cite{mt}, but the correct order of the mixing time is still unknown. 
For the case of lattice triangulations, which are triangulations of an $m \times n$ grid of points, no polynomial 
upper bound on the mixing time is known even when $m\geq2$ is kept fixed as $n\to\infty$.
The only known results in this case are limited to the weighted case~\cite{cmss15,cmss16,stauffer16}.
%
Another example of a related Markov chain that uses natural edge-flip type moves is the {\it switch Markov chain} for sampling from graphs with a given degree sequence. In this chain, at each iteration two random non-adjacent edges are removed and their four endpoints are randomly rematched; the move is rejected if it results in a multiple edge. Again, in the general case the mixing time of this Markov chain is unknown, though polynomial upper bounds exist when certain restrictions are placed on the degree sequence \cite{cdg07, greenhill15}. 

For the case of rectangular tilings, results for the mixing time of the edge-flip Markov chain have been quite rare.
One important result was obtained for domino tilings, which are tilings of an $n\times n$ square by rectangles of dimensions $1 \times 2$ or $2 \times 1$.
In this case, the edge-flip Markov chain is known to mix in time polynomial in the number of dominoes, a result that heavily relies on the connection between domino tilings and 
random lattice paths~\cite{lrs, rt}. 

The case of dyadic tilings exhibits interesting asymptotic properties that have been studied by combinatorialists~\cite{lsv02,jrs02}.
Tilings in which all rectangles are dyadic, but may have different areas, 
have been used as a basis for subdivision algorithms to solve problems such as approximating singular algebraic curves \cite{bcgy12} and classifying data using decision trees \cite{ScottNowak06}. 
In both of these examples, the unit square is repeatedly subdivided into smaller and smaller dyadic rectangles until the desired approximation or classification is achieved, 
with more subdivisions in the areas of the most interest (e.g., near the algebraic curve or where data classified differently is close together).

\vspace{1mm}
{\bf Proof ideas.}
We identify a certain block structure on dyadic tilings that allows us to relate the spectral gap of the edge-flip Markov chain to that of another, simpler Markov chain.
In the simpler Markov chain, which we refer to as the block dynamics, for each transition a large region of the tiling is selected and retiled uniformly at random, if possible.  
At the smallest scale, $n=4$, these correspond to exactly the moves of the (lazy) edge-flip Markov chain. 
The structure of these block moves allows us to set up a recursion that relates the spectral gap of the edge-flip Markov chain for tilings of size $n$ with that 
of sizes smaller than $n$ and that of the block dynamics.
This produces an inverse polynomial lower bound on the spectral gap of the edge-flip Markov chain. 

Specifically, 
we adapt a bisection approach inspired by spin system analysis~\cite{Martinelli1999,Cesi2001}. 
We bound the spectral gap $\gamma_k$ of the Markov chain $\Mn$ for dyadic tilings of size $n = 2^k$ by the product of the spectral gap 
$\gamma_{block}$ of the block dynamics Markov chain and the spectral gap $\gamma_{k-1}$ of $\M_{k-1}$, 
and then use recursion to obtain $\gamma_k \geq (\gamma_{block})^k = (\gamma_{block})^{\log n}$. 
As $\gamma_{block}$ is constant, this implies a polynomial relaxation time and thus a polynomial mixing time. 



To establish the explicit upper bound in Theorem~\ref{thm:upperbound}, we use a coupling argument to bound~$\gamma_{block}$; see, e.g., Chapter 13 of~\cite{lpw}.  
The distance metric we use is a carefully weighted average of two different notions of distance between tilings. 
We do a case analysis and show this distance metric always contracts by a factor of at least $1-1/17$ in each step, 
implying the spectral gap $\gamma_{block}$ is at least~$1/17$. 

We use a distinguishing statistic to show the mixing time and relaxation time of the edge-flip Markov chain for dyadic tilings are at least $\Omega(n^{1.38})$; 
again, see Chapter 13 of~\cite{lpw}. 
That is, we define a specific function $f$ on the state space of all dyadic tilings of size $n = 2^k$. 
By considering the variance and Dirichlet form of $f$, and using combinatorial properties of dyadic tilings, 
we can give an upper bound on the spectral gap and thus a lower bound on the relaxation and mixing times.

\section{Background}\label{sec:background}

Here we present some necessary information on dyadic tilings, including their asymptotic behavior, and on Markov chains, including mixing time and local variance. 

\subsection{Dyadic Tilings}

A {\it dyadic interval} is an interval that can be written in the form $[a2^{-s}, (a+1)2^{-s}]$ for non-negative integers $a$ and $s$ with $0 \leq a < 2^{s}$. A {\it dyadic rectangle} is the product of two dyadic intervals. 
A {\it dyadic tiling of size $n = 2^k$} is a tiling of the unit square by $n$  dyadic rectangles of equal area $1/n = 2^{-k}$ that do not overlap except on their boundaries; see Figure~\ref{fig:dyadic_ex}. Let $\Omega_k$ be the set of all dyadic tilings of size $n = 2^k$. 


We say a dyadic tiling has a {\it vertical bisector} if the line $x = 1/2$ does not intersect the interior of any dyadic rectangle in the tiling. We say it has a {\it horizontal bisector} if the same is true of the line $y = 1/2$. It is easy to prove that every dyadic tiling of size $n > 1$  has a horizontal bisector or a vertical bisector. 

The asymptotics of dyadic tilings were first explored by Lagarias, Spencer, and Vinson \cite{lsv02}, and we present a summary of their results.  
Let $A_k=|\Omega_k|$ denote the number of dyadic tilings of size $n = 2^k$. The unit square is the unique dyadic tiling consisting of one dyadic rectangle, so $A_0 = 1$.  There are two dyadic tilings of size $2$, since the unit square may be divided by either a horizontal or vertical bisector,  so $A_1 = 2$. One can also observe that $A_2 = 7$, $A_3 = 82$, $A_4 = 11047$, ... . In fact, the values $A_k$ can be shown to satisfy the recurrence $A_k = 2 A_{k-1}^2 - A_{k-2}^4$; we include a proof of this fact as presented in \cite{jrs02}, because we will use these ideas later. 

\begin{prop}[\cite{lsv02}] \label{thm:recurrence}
For $k \geq 2$, the number of dyadic tilings of size $2^k$ is $A_k = 2 A_{k-1}^2 - A_{k-2}^4$. 
\end{prop}
\begin{proof}
A dyadic tiling of size $ 2^k$ has a horizontal bisector, a vertical bisector, or both. If it has a vertical bisector, the number of ways to tile the left half of the unit square is $A_{k-1}$; by mapping $x \rightarrow 2x$, we can see that the left half of a dyadic tiling of size $2^k$ is equivalent to a dyadic tiling of the unit square of size $2^{k-1}$ because dyadic rectangles scaled by factors of two remain dyadic. Similarly, mapping $x \rightarrow 2x-1$, the right half of a dyadic tiling of size $2^k$ is equivalent to a dyadic tiling of size $2^{k-1}$. We conclude the number of dyadic tilings of size $2^k$ with a vertical bisector is $A_{k-1}^2$. Similarly, by appealing to  the maps $y \rightarrow 2y$ and $y\rightarrow 2y-1$, we conclude the number of dyadic tilings of size $2^k$ with a vertical bisector is $A_{k-1}^2$.  The number of dyadic tilings of size $2^k$ with both a horizontal and a vertical bisector is $A_{k-2}^4$, as each quadrant of any such tiling is equivalent to a dyadic tiling of the unit square of size $2^{k-2}$. This follows from appealing to the map $(x,y) \rightarrow (2x,2y)$ for the lower left quadrant, and appropriate translations of this for the other three quadrants. 
Altogether, we see $A_k = A_{k-1}^2 + A_{k-1}^2 -  A_{k-2}^4 = 2 A_{k-1}^2 - A_{k-2}^4$. 
\end{proof}

It is believed this recurrence does not have a closed form solution. We note that, as proved in~\cite{lsv02}, $A_k \sim \phi^{-1} \omega^{2^k} = \phi^{-1} \omega^{n}$, where $\phi = (1+\sqrt{5})/2$ is the golden ratio and $\omega = 1.84454757...$; an exact value for $\omega$ is not known.

We now define a recurrence for another useful statistic. We say that a dyadic tiling has a {\it left half-bisector} if the straight line segment from $(0,1/2)$ to $(1/2, 1/2)$  doesn't intersect the interior of any dyadic rectangles. Figure~\ref{fig:dyadic_ex}(a) does not have a left half-bisector, while Figure~\ref{fig:dyadic_ex}(b) does. We are interested in the number of ways to tile the left half of a vertically-bisected dyadic tiling of size $2^k$ such that it has a left half-bisector. 
Appealing to the dilation maps defined in the proof of Proposition~\ref{thm:recurrence}, this number is $A_{k-2}^2$. Among all possible ways to tile the left half of a vertically-bisected tiling $\sigma \in \Omega_k$, we define $f_k$ to be the fraction with a left half-bisector. We see 
  \begin{align*}
  f_k = \frac{A_{k-2}^2}{A_{k-1}}.
  \end{align*}
  We can similarly define {\it right half-bisectors}, {\it top half-bisectors}, and {\it bottom half-bisectors} by considering the straight line segments between $(1/2,1/2)$ and, respectively, $(1,1/2)$, $(1/2,1)$, and $(1/2,0)$. 
  Then $f_k$ is also the fraction of tilings of the right half of vertically-bisected tiling $\sigma$ with a right half-bisector, or the fraction of tilings of the top or bottom halves of a horizontally-bisected tiling $\sigma$ with a top or bottom half-bisector, respectively. 
  Note $f_2 = 0.5$, $f_3 = 4/7 \sim 0.571$, and $f_4 = 49/82 \sim 0.598$. We now examine the asymptotic behavior of~$f_k$. 
 
 \begin{lem} \label{lem:fk_recurrence}
 For all $k\geq3$, $f_k = \frac{1}{2-f_{k-1}^2}$.
 \end{lem}
 \begin{proof}
This follows from the recurrence for $A_k$ given in Proposition~\ref{thm:recurrence}:
\begin{align*}
f_k = \frac{A_{k-2}^2}{A_{k-1}} = \frac{A_{k-2}^2}{2A_{k-2}^2 - A_{k-3}^4} = \frac{1}{2-\frac{A_{k-3}^4}{A_{k-2}^2} } = \frac{1}{2-f_{k-1}^2}. 
\end{align*}

\vspace{-10mm}
 \end{proof}
 
 \noindent We can use this recurrence to study the asymptotic behavior of the sequence $\{f_k\}_{k=2}^{\infty}$. 
 
  \begin{lem}\label{lem:fk_limit} The sequence $\{f_k\}_{k=2}^{\infty}$ is strictly increasing and bounded above by $(\sqrt{5}-1)/2$.
 Furthermore, $\lim_{k\rightarrow \infty} f_k = (\sqrt{5}-1)/2$.  
 \end{lem}
 \begin{proof}
 Note $f_2 = 0.5 < (\sqrt{5}-1)/2$. Suppose by induction that $f_{k-1} < \frac{\sqrt{5}-1}{2}$. Then 
 \begin{align*}
f_{k} = \frac{1}{2-f_{k-1}^2} < \frac{1}{2-\left(\frac{\sqrt{5}-1}{2}\right)^2} = \frac{4}{8-(6-2\sqrt{5})} = \frac{4}{2+2\sqrt{5}} = \frac{2}{1+\sqrt{5}} = \frac{\sqrt{5}-1}{2}.\end{align*}

 To show $f_{k} < f_{k+1}$ for all $k \geq 3$, it suffices to show $x < 1/(2-x^2)$ for all $x \in\left[0.5, (\sqrt{5}-1)/2\right)$. This is equivalent to showing the polynomial $x^3-2x+1$ is positive in that range. Factoring shows this polynomial has roots at $1$, $(\sqrt{5}-1)/2$, and $-(\sqrt{5}+1)/2$, and is positive in the range $\left( -(\sqrt{5}+1)/2,(\sqrt{5}-1)/2\right)$. This implies $f_{k} < f_{k+1}$, so the sequence is strictly increasing. 

The sequence $\{f_k\}_{k=2}^\infty$ is bounded and monotone, so it must converge to some limit $\beta$. To find $\beta$, we consider the function $g(x) = 1/(2-x^2)$, which is the recurrence for the $f_k$. This function is continuous away from $\sqrt{2}$ and $-\sqrt{2}$, and thus certainly is continuous on $\left[0.5, (\sqrt{5}-1)/2\right],$ the range of possible values for the $f_k$ and their limit $\beta$. This continuity implies $$g(\beta) = g\left(\lim_{k\rightarrow\infty} f_k\right) = \lim_{k\rightarrow\infty} g(f_k) = \lim_{k\rightarrow\infty} f_{k+1} = \beta.$$  
Thus the limit $\beta$ is necessarily a fixed point of $g(x)$. The fixed points of $g(x)$ are exactly  the three roots of $x^3-2x+1$ found above, and the only one in $\left[0.5, (\sqrt{5}-1)/2\right]$ is $(\sqrt{5}-1)/2$. We conclude $\lim_{k\rightarrow\infty} f_k =  (\sqrt{5}-1)/2$, as desired. 
 \end{proof}

\subsection{Markov Chains}\label{sec:markovchains}

We will consider only discrete time Markov chains in this paper, though identical results also hold for the analogous continuous time Markov chains.
 Any finite ergodic Markov chain is known to converge to a unique stationary distribution $\pi$. The time a Markov chain with transition matrix $P$ takes to converge to its stationary distribution is measured by the {\it total variation distance}, which captures how far the distribution after $t$ steps is from the stationary distribution given a worst case starting configuration: 
 \begin{align*}
 \| P^t - \pi\|_{{\rm TV}} = \max_{x\in\Omega} \frac{1}{2} \sum_{y\in \Omega} |P^t(x,y)-\pi(y)|.
 \end{align*}
 The mixing time of a Markov chain $\M$ is defined to be 
 \begin{align*} 
\tmix(\varepsilon) = \min \{t : \|P^{t'} - \pi\|_{{\rm TV}} \leq \varepsilon\ \  \forall \  t' \geq t\}.
\end{align*}
For convenience, as is standard we define $\tmix=\tmix(1/4)$.

We will bound the mixing time of the edge-flip Markov chain for dyadic tilings by studying its relaxation time and spectral gap. The {\it spectral gap} $\gamma$ of a Markov chain $\M$ with transition matrix $P$ is $1-\lambda_2$, where $\lambda_2$ is the second largest eigenvalue of $P$. 
For a lazy Markov chain $\M$, the relaxation time, denoted by $\trel$, is then the inverse of this spectral gap; we will see in the next section that the edge-flip Markov chain for dyadic tilings, as we've defined it, is lazy. 
The following well-known proposition relates the relaxation time and mixing time for  Markov chains; for a proof, see, e.g., \cite[Theorem 12.3 and Theorem 12.4]{lpw}. 

\begin{prop}\label{prop:rel-mix}
Let $\M$ be an ergodic Markov chain on state space $\Omega$ with reversible transition matrix $P$ and stationary distribution $\pi$. Let $\pi_{min}  = \min_{x \in\Omega} \pi(x)$. Then: 
\begin{align*}
(\trel-1)  \log\left(\frac{1}{2\varepsilon}\right)   \leq \tmix(\varepsilon) \leq \log \left(\frac{1}{\varepsilon \pi_{min}}\right) \trel.
\end{align*}

\end{prop}

We will bound the spectral gap, and thus the relaxation and mixing times, of the edge-flip Markov chain for dyadic tilings by considering functions on the chain's state space. For $f:\Omega\rightarrow \mathbb{R}$, the {\it variance} of $f$ with respect to a distribution $\pi$ on $\Omega$ can be expressed as: 
\begin{align*}
   \Var_\pi(f) = \sum_{x\in \Omega} \pi(x) \left(f(x) - \mathbb{E}_\pi[f(x)]\right)^2 = \frac{1}{2}\sum_{x,y\in\Omega} \pi(x) \pi(y) (f(x) - f(y))^2.
\end{align*}
 We will only be considering the variance with respect to the uniform distribution on $\Omega$, so the subscript $\pi$ will be omitted. 
 For a given reversible transition matrix $P$ on state space $\Omega$ with stationary distribution $\pi$, the {\it Dirichlet form}, also know as the {\it local variance}, associated to the pair $(P,\pi)$ is,  for any function $f:\Omega\rightarrow\mathbb{R}$,
 \begin{align*}
 \mcE(f) = \frac{1}{2} \sum_{x,y \in \Omega} [f(x)-f(y)]^2 \pi(x) P(x,y).
  \end{align*} 
   As we see in the following well-known proposition, the Dirichlet form and variance of a function $f$ can be used to bound the spectral gap of a transition matrix, and therefore the relaxation time and mixing time of a Markov chain; see, e.g., \cite[Lemma 13.12]{lpw}. 
  \begin{prop}\label{prop:spectralbound}
Given a Markov chain with reversible transition matrix $P$ and stationary distribution $\pi$, the spectral gap $\gamma = 1-\lambda_2$ of $P$ satisfies
\begin{align*}
\gamma = \min_{\substack{f: \Omega\rightarrow \mathbb{R} \\ \Var_\pi(f) \neq 0}} \frac{\mcE(f) }{ \Var_\pi(f)}. 
\end{align*}
  \end{prop}
 \noindent We will use this proposition in both our upper bound and lower bound proofs.

\section{The Edge-Flip Markov Chain $\Mn$}

     Let $n = 2^k$. For $k \geq 1$, the edge-flip  Markov chain $\Mn$ on the state space $\Omega_k$ of all dyadic tilings of size $2^k$ is given by the following rules. \vspace{2mm} \\ 
 Beginning at any $\sigma_0 \in \Omega_k$, repeat:
   \begin{itemize}
  \item Choose a rectangle $R$ of $\sigma_i$ uniformly at random.
  \item Choose $left$, $right$, $top$, or $bottom$ uniformly at random; let $e$ be the corresponding side of $R$.
  \item If $e$ bisects a rectangle of area $2^{-k+1}$, remove $e$ and replace it with its perpendicular bisector  to obtain $\sigma_{i+1}$ if the result is a valid dyadic tiling; else, set $\sigma_{i+1} = \sigma_i$. 
  \end{itemize}
  An example of an edge-flip move of $\Mn$ is shown in Figure~\ref{fig:edgeflip_ex}(a); two selections of $R$ and $e$ that do not yield valid moves are shown in (b) and (c). Let $P_{k,edge}$ denote the transition matrix of this edge-flip Markov chain and $\gamma_k$ its spectral gap. For every valid edge flip, there are two choices of $(R,e)$ that result in that move. This implies every move between two tilings differing by an edge flip occurs with probability $1/(2n) = 2^{-k-1}$, so all off-diagonal entries of $P_{k,edge}$ are either $2^{-k-1}$ or $0$.
  
   \begin{figure} \centering
  \includegraphics[scale = 0.65]{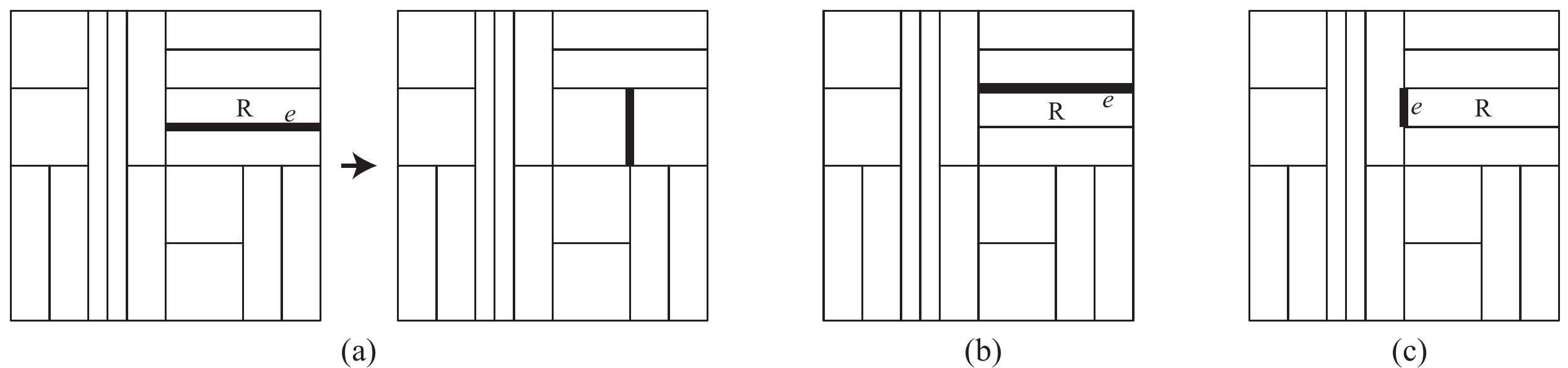}
  \caption{A random rectangle $R$ and one of its edges $e$ are selected in each iteration of $\Mn$. (a) Random choices of $R$ and $e$ as shown  yield a valid edge flip. (b) Random choices of $R$ and $e$ as shown do not yield a valid edge flip  as flipping edge $e$ results in a tiling that is not dyadic. (c) Random choices of $R$ and $e$ as shown do not yield a valid edge flip as flipping edge $e$ does not produce a tiling of the unit square by rectangles. } 
  \label{fig:edgeflip_ex}
  \end{figure}
  
  The Markov chain $\Mn$, in a slightly different form, was introduced by Janson, Randall and Spencer~\cite{jrs02}. Note that $\Mn$
  is lazy, as for any rectangle $R$ of a dyadic tiling at most one of its left and right edges can be flipped to produce another valid dyadic tiling. This is because if $R$'s projection onto the $x$-axis is dyadic interval $ [a2^{-s}, (a+1)2^{-s}]$ for $a,s \in \mathbb{Z}_{\geq 0}$, then flipping its left edge yields a rectangle with $x$-projection $[(a-1)2^{-s}, (a+1)2^{-s}]$ and flipping its right edge yields a rectangle with $x$-projection $[a2^{-s}, (a+2)2^{-s}]$.  If $a$ is even, the first of these intervals is not dyadic, while if $a$ is odd, the second is not, so at most one of these edge flips produces a valid dyadic tiling. Similarly, at most one of $R$'s top and bottom edges yields a valid edge flip. This implies in each iteration with probability at least $1/2$ a pair $(R,e)$ is selected that does not yield a valid edge flip move. 
  
  It was previously shown that this Markov chain is irreducible in \cite{jrs02}, so $\Mn$ is ergodic and thus has a unique stationary distribution. The uniform distribution satisfies the detailed balance equation, implying both that $\Mn$ is reversible and that its stationary distribution is uniform on~$\Omega_k$. 
  
  While we index this edge-flip Markov chain for dyadic tilings of size $n = 2^k$ by $k$ instead of by~$n$, it is important to keep in mind that we wish to show the mixing time of $\Mn$ is polynomial in $n$, not polynomial in $k$.

\subsection{The Block Dynamics Markov Chain $\Mnb$}\label{sec:block}

To analyze the mixing time of Markov chain $\Mn$, we will appeal to a similar Markov chain that uses larger block moves instead of single edge flips. We use in a crucial way the bijection between tilings in $\Omega_{k-1}$ and the left or right (resp. top or bottom) half of a tiling in $\Omega_k$ that has a vertical (resp. horizontal) bisector, as discussed in the proof of Proposition~\ref{thm:recurrence}.

     For $k \geq 2$, the block dynamics  Markov chain $\Mnb$ on the state space $\Omega_k$ of all dyadic tilings of size $2^k$ is given by the following rules.
  \vspace{2mm} \\ 
 Beginning at any dyadic tiling $\sigma_0$, repeat:
   \begin{itemize}
     \item Uniformly at random choose a tiling $\rho \in \Omega_{k-1}$. 
  \item Uniformly at random choose $Left$, $Right$, $Top$, or $Bottom$.
  \item To obtain $\sigma_{i+1}$:
  \begin{itemize}
  \item If $Left$ was chosen and $\sigma$ has a vertical bisector, retile $\sigma$'s left half with $\rho$, under the mapping $x \rightarrow x/2$.
  \item  If $Right$ was chosen and $\sigma$ has a vertical bisector, retile $\sigma$'s right half with $\rho$, under the mapping $x \rightarrow (x+1)/2$.
    \item If $Bottom$ was chosen and $\sigma$ has a horizontal bisector, retile $\sigma$'s bottom half with $\rho$, under the mapping $y \rightarrow y/2$.
  \item  If $Top$ was chosen and $\sigma$ has a horizontal bisector, retile $\sigma$'s top half with $\rho$, under the mapping $y \rightarrow (y+1)/2$.
  \end{itemize}
  \item Else, set $\sigma_{i+1} = \sigma_i$. 
  \end{itemize}
  Let $P_{k,block}$ be the transition matrix of this Markov chain and let $\gamma_{k,block}$ be its spectral gap. Any valid nonstationary transition of $\Mnb$ occurs with probability $1/(4|\Omega_{k-1}|)$. 
   This Markov chain is not lazy, but it is aperiodic, irreducible, and reversible.  This implies it is ergodic and thus has a unique stationary distribution, which by detailed balance is uniform on  $\Omega_k$.

\section{A Polynomial upper bound on the mixing time of $\Mn$}

Recall we wish to show the mixing time of $\Mn$ is polynomial in $n = 2^k$, not polynomial in $k$.
We show the spectral gap $\gamma_k$ of $\Mn$ and the spectral gap $\gamma_{k-1}$ of $\M_{k-1}$ differ by a multiplicative constant (specifically, $1/17$) by appealing to the Dirichlet forms of both of these Markov chains as well as the block dynamics Markov chain $\Mnb$. We can then use recursion to show $\gamma_k$ is bounded below by $(1/17)^k$, which, because $k = \log n$, gives a polynomial upper bound on the relaxation time and thus on the mixing time of $\Mn$. 

For any function $f:\Omega_k\rightarrow \mathbb{R}$, we will denote the Dirichlet form of $f$ with respect to transition matrix $P_{k,edge}$ and the uniform stationary distribution as $\mcE_{k,edge}(f)$. The Dirichlet form of $f$ with respect to transition matrix $P_{k,block}$ and the uniform stationary distribution will be $\mcE_{k,block}(f)$. 
We will let the variance of function $f$ on $\Omega_k$ with respect to the uniform stationary distribution be $\Var_k(f)$. Here the $k$ indicates which state space $\Omega_k$  we are considering, rather than which distribution on $\Omega_k$ the variance is taken with respect to; all variances we consider will be with respect to the uniform distribution. 

Because we consider two different Markov chains on the same state space $\Omega_k$, there are two different notions of adjacencies on this state space, each corresponding to the moves of one of these Markov chains. For $x,y$ in $\Omega_k$, we say $x \sim_e y$ if $x$ and $y$ differ by a single edge flip move of $\Mn$ and $x\sim_by$ if $x$ and $y$ differ by a single move of the block dynamics chain $\Mnb$. More specifically, if $x$ and $y$ differ by a retiling of their left half (implying $x$ and $y$ both have a vertical bisector and are the same on their right half), we say $x \sim_L y$; then $x \sim_R y$, $x \sim_T y$, and $x \sim_B y$ are defined similarly for the right, top, and bottom halves.

\begin{thm} \label{thm:spectral-bound}
For any $k \geq 2$, the spectral gap $\gamma_k$ of the edge-flip Markov chain $\Mn$ satisfies
\begin{align*}\gamma_k \geq \gamma_{k,block}\cdot \gamma_{k-1}\end{align*}
\end{thm}
\begin{proof}
We begin by computing the Dirichlet forms for block dynamics and then for the edge-flip dynamics, which will allow comparison of their spectral gaps. Recall that for any function $f:\Omega_k \rightarrow \bbR$, 
\begin{align*}
\mcE_{k,block}(f) = \frac{1}{2} \sum_{x\sim_b y \in \Omega_k} \pi(x)P_{k,block}(x,y) \left(f(x)-f(y)\right)^2 .
\end{align*}
This sum can be split up into four terms, depending on whether $x$ and $y$ differ by a retiling of their left, right, top, or bottom halves. We now analyze the first of these terms, containing all pairs $x,y$ differing by a retiling of their left halves. For $x_L,x_R \in \Omega_{k-1}$, by $x_Lx_R$ below we mean the tiling in $\Omega_{k}$ with a vertical bisector whose left half is $x_L$ under  the map $x \rightarrow x/2$ and whose right half is $x_R$ under  the map $x \rightarrow (x+1)/2$. 
\begin{align*}
\mcE^L_{k,block} 
&= \frac{1}{2} \sum_{x \sim_L y}   \frac{1}{|\Omega_k|} \frac{1}{4|\Omega_{k-1}|} (f(x) - f(y))^2
\\&= \frac{1}{8} \sum_{x_R \in \Omega_{k-1}} \sum_{x_L, y_L \in \Omega_{k-1}}   \frac{1}{|\Omega_k|} \frac{1}{|\Omega_{k-1}|} (f(x_Lx_R) - f(y_Lx_R))^2
\\&= \frac{1}{4}  \sum_{x_R \in \Omega_{k-1}} \frac{|\Omega_{k-1}|}{|\Omega_k|}\left( \frac{1}{2} \sum_{x_L, y_L \in \Omega_{k-1}} \frac{1}{|\Omega_{k-1}|^2} (f(x_Lx_R) - f(y_L x_R))^2 \right).
\end{align*}
For each $x_R \in \Omega_{k-1}$, the function $f|_{ x_R}: \Omega_{k-1} \rightarrow \bbR $ given by $f|_{x_R}(z) = f(z x_R)$ has variance $\Var_{k-1}(f|_{x_R})$ (with respect to the uniform distribution) 
that is exactly equal to the term in parentheses above. 
By appealing to Proposition~\ref{prop:spectralbound}, we can bound this variance using both the Dirichlet form of function $f|_{x_R}$ associated to transition matrix $P_{k-1,edge}$ and the spectral gap $\gamma_{k-1}$ of this Markov chain $\M_{k-1}$. Thus, 
\begin{align*}
\mcE^L_{k,block} = \frac{1}{4}  \sum_{x_R \in \Omega_{k-1}} \frac{|\Omega_{k-1}|}{|\Omega_k|} \Var_{k-1}(f|_{x_R})
& \leq \frac{1}{4}  \sum_{x_R \in \Omega_{k-1}} \frac{|\Omega_{k-1}|}{|\Omega_k|} \frac{\mcE_{k-1,edge}(f|_{ x_R})}{\gamma_{k-1}}
\end{align*}
We now see that the Dirichlet form for the edge-flip Markov chain on $\Omega_{k-1}$ is 
\begin{align*}
\mcE_{k-1,edge}(f|_{x_R}) &= \frac{1}{2} \sum_{\substack{x_L,y_L \in \Omega_{k-1} \\  x_{L} \sim_e y_L}} \pi(x_L) P(x_L,y_L) \left(f(x_Lx_R) - f(y_Lx_R)\right)^2
\\&= \sum_{ \substack{x_L,y_L \in \Omega_{k-1} \\  x_{L} \sim_e y_L}} \frac{1}{|\Omega_{k-1}|} \frac{1}{2n}\left(f(x_Lx_R) - f(y_Lx_R)\right)^2\end{align*}
Using this expression, we see that
\begin{align*}
\mcE^L_{k,block} (f) &\leq \frac{1}{4\gamma_{k-1}}  \sum_{x_{ R} \in \Omega_{k-1}} \frac{|\Omega_{k-1}|}{|\Omega_k|}  \left( \sum_{\substack{x_L,y_L \in \Omega_{k-1} \\  x_{L} \sim_e y_L}}  \frac{1}{|\Omega_{k-1}|} \frac{1}{2n}\left(f(x_{\small L} x_{\small R}) - f(y_Lx_R)\right)^2\right)
\\&= \frac{1}{4\gamma_{k-1}}  \sum_{\substack{ x,y \in \Omega_{k} \\ x\sim _e y \\ x \sim_{\tiny L}  y}}
\frac{1}{|\Omega_{k}|} \frac{1}{2n}\left(f(x) - f(y)\right)^2.
\end{align*}
We now compare this to the Dirichlet form  for the edge flip Markov chain on $\Omega_k$, which we recall is 
\begin{align*}
\mcE_{k} (f)= \frac{1}{2} \sum_{\substack{ x,y \in \Omega_{k} \\ x\sim _e y }}\frac{1}{|\Omega_{k}|} \frac{1}{2n}\left(f(x) - f(y)\right)^2.
\end{align*}
We note for every $x,y \in\Omega_k$ such that $x \sim_e y$, at least one of and at most two of $x\sim_Ly$, $x\sim_Ry$, $x\sim_Ty$, and $x\sim_By$ hold. Thus each summand of $\mcE_k(f) $ appears at most twice as a summand of $$\mcE_{k,block} (f) = \mcE^L_{k,block} (f) +\mcE^R_{k,block} (f) +\mcE^T_{k,block} (f) +\mcE_{k,block}^B (f) .$$
It follows that
$$
\mcE_{k,block} (f) \leq  \frac{1}{4\gamma_{k-1}} \cdot 2 \cdot \left( 2 \mcE_k(f)\right) =\frac{\mcE_{k,edge}(f)}{\gamma_{k-1}}.
$$
Note this implies that for any $f$,
$$
\Var_k(f) \leq \frac{\mcE_{k,block}(f)}{\gamma_{k,block} } \leq \frac{\mcE_{k,edge}(f)}{\gamma_{k,block} \cdot\gamma_{k-1}}.
$$
Let $f$ be chosen to be the function achieving equality in 
$$\Var_k(f) \leq \frac{\mcE_{k,edge} (f)}{\gamma_{k}}.$$
We conclude 
$$
\gamma_k = \frac{\mcE_{k,edge}(f)}{\Var_k(f)} \geq \gamma_{k,block}\cdot \gamma_{k-1}. 
$$

\end{proof}

We will prove in Section~\ref{sec:block-constant} that the spectral gap of the block dynamics Markov chain is at least 1/17 for sufficiently large $k$. This can be used to bound the spectral gap, the relaxation time, and finally the mixing time of $\Mn$. 

\begin{thm}\label{thm:block17}
   There exists a positive integer $k_0$ such that for all 
   $k \geq k_0$, the spectral gap $\gamma_{k,block}$ is at least $1/17$.
\end{thm}
\begin{proof}
See Section~\ref{sec:block-constant}. We introduce a distance metric on dyadic tilings, and then give a coupling where the distance between two tilings decreases in expectation after one iteration by a multiplicative factor of $1-\frac{1}{17}$ for all $k$ sufficiently large. By a result of Chen~\cite{Chen}, this implies the~theorem.
\end{proof}


\noindent We are now ready to prove our first main theorem, Theorem~\ref{thm:upperbound}, which states that the relaxation time of $\Mn$ for $n = 2^k$ is $O(n^{\log 17})$ and its mixing time is $O(n^{1+\log 17})$

\begin{proof}[Proof of Theorem~\ref{thm:upperbound}]
By Theorems~\ref{thm:spectral-bound} and~\ref{thm:block17}, the spectral gap of $\Mn$ satisfies
\begin{align*}
   \gamma_k \geq  \frac{1}{17} \gamma_{k-1} 
   \geq 17^{-(k-k_0)} \gamma_{k_0}, 
\end{align*}
where $k_0$ is the value from Theorem~\ref{thm:block17}.
Since $\gamma_{k_0}$ is a constant that does not depend on $n$, we obtain
$$
   \gamma_k = \Omega\left( 17^{-k}\right)= \Omega\left( n^{-\log 17}\right) = \Omega\left( n^{-4.09}\right).
$$
Because $\Mn$ is a lazy Markov chain, its relaxation time satisfies 
$$
   \trel = O\left(n^{\log 17}\right).
$$
To use this to bound the mixing time of $\Mn$, we appeal to Proposition~\ref{prop:rel-mix}, though we first must calculate $\pi_{min}$. For $\pi$ the uniform distribution, ${\min_{x \in \Omega_k} \pi(x) = 1/|\Omega_k|}$. By the asymptotics of dyadic tilings, a loose bound is ${1/\pi_{min} = |\Omega_k| < 2^n} $. This implies
\begin{align*}
   \tmix = O\left(n^{1+\log 17}\right).
\end{align*}
\end{proof}

\section{Lower bound on the mixing time of $\M_n$}\label{sec:lowerbound}
In this section we give the proof of Theorem~\ref{thm:lowerbound}.
%
For this, we define the following subsets of $\Omega_k$:
\begin{align*}
   \Omega_k^+ &= \left\{x\in\Omega_k \colon \text{ $x$ has both a horizontal and a vertical bisector}\right\}, \\
   \Omega_k^| &= \left\{x\in\Omega_k \colon \text{ $x$ has a vertical bisector}\right\}, \text{ and} \\
   \Omega_k^- &= \left\{x\in\Omega_k \colon \text{ $x$ has a horizontal bisector}\right\}.
\end{align*}
By definition, we have $\Omega_k^+ = \Omega_k^| \cap \Omega_k^-$.
We start with the following simple lemma.
\begin{lemma}\label{lem:plus}
   For all $k\geq 2$, we have
   $$
      \frac{|\Omega_k|}{|\Omega_k^+|}=\frac{2}{f_k^2}-1\geq 2\phi+1. 
   $$
   Furthermore, $\lim_{k\to\infty}\frac{|\Omega_k|}{|\Omega_k^+|}=2\phi+1
   $, where $\phi=\frac{\sqrt{5}+1}{2}$ is the golden ratio.
\end{lemma}
\begin{proof}
   Using that $|\Omega_k^+|=|\Omega_{k-2}|^4$, and Proposition~\ref{thm:recurrence}, we have
   $$
      \frac{|\Omega_k|}{|\Omega_k^+|}
      = \frac{2|\Omega_{k-1}|^2-|\Omega_{k-2}|^4}{|\Omega_{k-2}|^4}
      =\frac{2}{f_k^2}-1.
   $$
   By Lemma~\ref{lem:fk_limit},  $f_k \leq \frac{\sqrt{5}-1}{2}=\frac{1}{\phi}=\lim_{k\to\infty}f_k$. This, along with the identity $\phi^2 = 1+\phi$, implies the lemma. 
\end{proof}

We will also require the following technical estimate.
\begin{lemma}\label{lem:upsilon}
   For any $k\geq 2$, we have 
   $$
      \frac{1}{|\Omega_k|} \prod_{i=0}^{k-2} |\Omega_{i}|^2 
      \leq \phi^{-2k+2}
   $$
\end{lemma}
\begin{proof}
   We will show how to estimate $\prod_{i=0}^{k-2} |\Omega_{i}|^2 $ via the construction of a tiling in $\Omega_k$.
   We start with a tiling with both a horizontal and a vertical bisector, as in Figure~\ref{fig:constr}(a). 
   Then we inductively do the following. 
   Both quadrants of the left half are tiled independently with a uniformly random tiling from $\Omega_{k-2}$.
   In the top-right quadrant, we add a vertical bisector and complete the two halves of this quadrant with independent, uniformly random tilings from $\Omega_{k-3}$. 
   Finally, in the bottom-right quadrant, we create a horizontal and a vertical bisector, reaching the tiling in Figure~\ref{fig:constr}(b).
   Then we take this bottom-right quadrant, and iterate the procedure above; see Figure~\ref{fig:constr}(c,d) for the configurations after one and two more iterations.
   \begin{figure}
      \begin{center}
         \includegraphics[scale=.8]{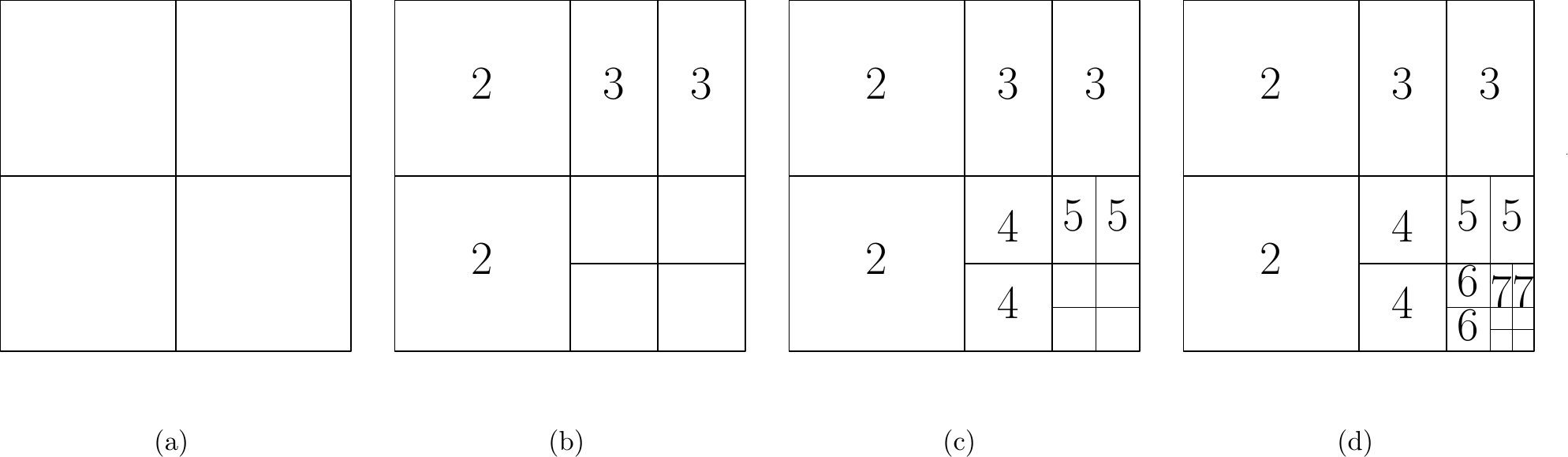}
      \end{center}\vspace{-.5cm}
      \caption{The construction of a tiling to count $\prod_{i=0}^{k-2} |\Omega_i|^2$. 
         A rectangle with number $a$ indicates that we tile it with a tiling from $\Omega_{k-a}$.}
      \label{fig:constr}
   \end{figure}
   This iteration continues until creating a bisector will result in rectangles of area less than $2^{-k}$. In the case where an attempt is made to divide a rectangle of area $2^{-k+1}$ into four rectangles of equal area by adding both a horizontal and vertical bisector, we instead add just a horizontal bisector, resulting in two rectangles each of area $2^{-k}$.

   Let $\Upsilon_k\subset \Omega_k$ be the set of tilings obtained in this way.
   Note that the number of tilings in $\Upsilon_k$ is exactly $\prod_{i=0}^{k-2} |\Omega_i|^2$.
   Since $\Upsilon_k\subset \Omega_k^+$, we have that 
   $\frac{|\Upsilon_k|}{|\Omega_k|}\leq \frac{|\Omega_k^+|}{|\Omega_k|}$, where the first expression is exactly the value we wish to bound. Using the construction above until Figure~\ref{fig:constr}(b), we obtain that 
   $$
      \frac{|\Upsilon_k|}{|\Omega_k|}
      \leq \frac{|\Omega_k^+|}{|\Omega_k|} \frac{|\Omega_{k-2}^||}{|\Omega_{k-2}|},
   $$
   where the second factor stands for the fact that the top-right quadrant must contain a vertical bisector. 
   Iterating this in the bottom-right quadrant, we obtain
   \begin{align}\label{eqn:product}
      \frac{|\Upsilon_k|}{|\Omega_k|}
      \leq \frac{|\Omega_k^+|}{|\Omega_k|} \frac{|\Omega_{k-2}^||}{|\Omega_{k-2}|}\frac{|\Omega_{k-2}^+|}{|\Omega_{k-2}|} \frac{|\Omega_{k-4}^||}{|\Omega_{k-4}|}...
   \end{align}
   Proposition~\ref{thm:recurrence} gives that 
   $$
      \frac{|\Omega_k^||}{|\Omega_k|} 
      = \frac{|\Omega_k| + |\Omega_{k-2}|^4}{2|\Omega_k|}
      = \frac{1}{2}\left(1+\frac{|\Omega_{k}^+|}{|\Omega_k|}\right)
      \leq \frac{1}{2}\left(1+\frac{1}{2\phi+1
      }\right)
      = \frac{\phi^2
      }{2\phi+1
      },
   $$
   where the inequality follows from Lemma~\ref{lem:plus}.  
   For even $k$, because $|\Omega_0^|| = 0$ the last term we can obtain in  (\ref{eqn:product}) is $\frac{|\Omega_{2}^+|}{|\Omega_2|}$, so we can write
   \begin{align*}
    \frac{|\Upsilon_k|}{|\Omega_k|}
          \leq \left(\prod_{i=0}^{k/2-2}\frac{|\Omega_{k-2i}^+|}{|\Omega_{k-2i}|} \cdot\frac{|\Omega_{k-2i-2}^||}{|\Omega_{k-2i-2}|}  \right)\frac{|\Omega_{2}^+|}{|\Omega_2|}
          \leq \frac{1}{2\phi+1}\left( \frac{1}{2\phi+1} \cdot\frac{\phi^2}{2\phi+1}\right)^{\frac{k}{2}-1} 
          = \frac{\phi^{-2k+4}}{2\phi +1}
          \leq \phi^{-2k+2},
   \end{align*}
   where the last expressions come from, respectively, identities for $\phi$ and the easily-checked inequality $2\phi+1 >  \phi^2$.
   When $k$ is odd, the last term in (\ref{eqn:product}) is $\frac{|\Omega_1^||}{|\Omega_1|}$ because $|\Omega_1^+| = 0$, so we can write 
    \begin{align*}
    \frac{|\Upsilon_k|}{|\Omega_k|}
          \leq \left(\prod_{i=0}^{(k-3)/2}\frac{|\Omega_{k-2i}^+|}{|\Omega_{k-2i}|} \cdot\frac{|\Omega_{k-2i-2}^||}{|\Omega_{k-2i-2}|}  \right)
          \leq \left( \frac{1}{2\phi+1} \cdot\frac{\phi^2}{2\phi+1}\right)^{\frac{k-1}{2}} 
          \leq \phi^{-2k+2},
   \end{align*}
   where again the last expression is the result of applying identities for $\phi$ and simplifying. 
   %
\end{proof}

We are now ready to prove our second main theorem, giving a lower bound on the mixing and relaxation times of $\Mn$ of $\Omega(n^{2\log\phi})$. 

\begin{proof}[Proof of Theorem~\ref{thm:lowerbound}]
   We will derive a upper bound on the spectral gap $\gamma_k$. 
   To do this, we consider the test function $f:\Omega_k \to \{0,1\}$ such that 
   \begin{equation}
      f(x) \text{ is 1 if $x\in\Omega_k^|$, and 0 otherwise.}
      \label{eq:f}
   \end{equation}
   We will apply this function to the characterization of the spectral gap in Proposition~\ref{prop:spectralbound}.
   
   We start by showing that the variance of $f$ is bounded away from $0$ as $k\to\infty$.
   Recall that $\Var_k$ denotes variance with respect to the uniform measure on $\Omega_k$.
   \begin{claim}\label{claim:var}
      With $f:\Omega_k\to\{0,1\}$ as in~\eqref{eq:f}, we have that 
      $$
         \lim_{k\to\infty}\Var_k(f) = \sqrt{5}-2.
      $$
   \end{claim}
   \begin{proof}[Proof of claim]
      We start by writing
      \begin{equation}
         \Var_k(f) = \sum_{x\in \Omega_k^|}\sum_{y\in \Omega_k\setminus \Omega_k^|} \frac{1}{|\Omega_k|^2}
         = \frac{|\Omega_k^||\cdot |\Omega_k\setminus \Omega_k^||}{|\Omega_k|^2}.
         \label{eq:writevar}
      \end{equation}
      Since $|\Omega_k^||= |\Omega_{k-1}|^2$, using Proposition~\ref{thm:recurrence} we obtain
      \begin{equation}
         |\Omega_k^||=\frac{|\Omega_k|+|\Omega_{k-2}|^4}{2}=\frac{|\Omega_k|+|\Omega_k^+|}{2},
         \label{eq:step1}
      \end{equation}
      and
      \begin{equation}
         |\Omega_k\setminus \Omega_k^||
         =|\Omega_k|-|\Omega_k^||
         =\frac{|\Omega_k|-|\Omega_k^+|}{2}.
         \label{eq:step2}
      \end{equation}
      Plugging~\eqref{eq:step1} and~\eqref{eq:step2} into~\eqref{eq:writevar}, we get
      $$
         \Var_k(f)
         = \frac{1}{4}\left(1+\frac{|\Omega_k^+|}{|\Omega_k|}\right)\left(1-\frac{|\Omega_k^+|}{|\Omega_k|}\right)
         = \frac{1}{4}\left(1-\left(\frac{|\Omega_k^+|}{|\Omega_k|}\right)^2\right).
      $$
      Then Lemma~\ref{lem:plus} yields
      $$
         \lim_{k\to\infty}\Var_k(f)= \frac{1}{4}\left(1-\frac{1}{(2\phi+1)^2}\right).
      $$   
      Plugging in the value of $\phi$ completes the proof of the claim.
   \end{proof}
   
   Now it remains to obtain an upper bound for $\mcE(f)$.
   Let $\partial \Omega_k^|$ be the set of tilings in $\Omega_k\setminus \Omega_k^|$ which can be obtained from a tiling in $\Omega_k^|$ via one edge flip.
   Recall for two tilings $x,y\in\Omega_k$, we write $x\sim_e y$ if $x$ can be obtained from $y$ by one edge flip.
   Hence,
   $$
      \mcE(f)
      = \sum_{x \in \partial \Omega_k^|}\sum_{y\in\Omega_k^| \colon y \sim_e x} \frac{1}{|\Omega_k|} \frac{1}{2n}.
   $$
   Note that each tiling in $\partial \Omega_k^|$ has a horizontal bisector and is not in $\Omega_k^+$. 
   This means that it has exactly one edge flip that can bring it into $\Omega_k^|$, which is the flip that creates a 
   vertical bisector.
   Then, we have 
   $$
      \mcE(f)
      = \frac{|\partial \Omega_k^||}{2n \cdot |\Omega_k|}.
   $$
   Now we need to describe the set $\partial \Omega_k^|$. 
   It is a set of tilings with no vertical bisector, but with one edge flip that creates a vertical bisector; see Figure~\ref{fig:bd}. 
   \begin{figure}[htbp]
      \begin{center}
         \includegraphics{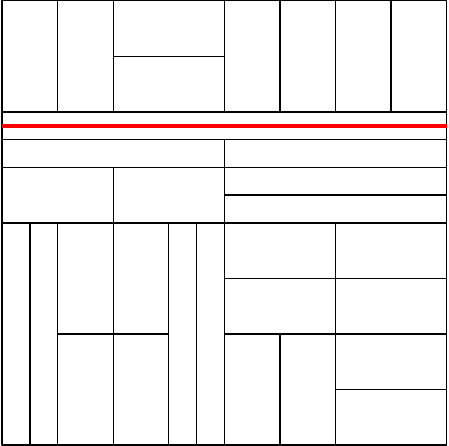}
      \end{center}\vspace{-.5cm}
      \caption{A tiling in $\partial \Omega_k^|$, with the red edge being the flip that brings the tiling into $\Omega_k^|$.}
      \label{fig:bd}
   \end{figure}
   Note that the edge whose flip creates a vertical bisector must be a horizontal edge of length $1$ 
   which flips to a vertical edge of length $2/n$. From now on we will refer to this edge as the 
   \emph{pivotal edge}.
   
   In order to estimate the cardinality of 
   $\partial \Omega_k^|$, we will describe a procedure to construct a tiling $x\in\partial \Omega_k^|$, observing the position of the pivotal edge.
   Note that $x$ must have a horizontal bisector, which splits $[0,1]^2$ into its top and bottom halves. 
   Assume that the pivotal edge is in the top half of $x$. This implies that the bottom half of $x$ must itself contain a vertical bisector
   since the pivotal edge must be the only edge that forbids a vertical bisector to exist, see Figure~\ref{fig:sum}(a).
   The two quadrants in the bottom half are simply any tilings of $\Omega_{k-2}$.
   Note also that the top half of $x$ must contain a horizontal bisector, otherwise $x \not\in \partial \Omega_k^|$, see Figure~\ref{fig:sum}(b).
   Then we iterate the above construction: among the two halves of the top half, one 
   must contain the pivotal edge, say the bottom one, while the other contains a vertical bisector, each side of which being completed with a tiling 
   from $\Omega_{k-3}$, which gives the configuration in Figure~\ref{fig:sum}(c). 
   Continuing this for $k-2$ steps concludes the construction.
   \begin{figure}
      \begin{center}
         \includegraphics[scale=.8]{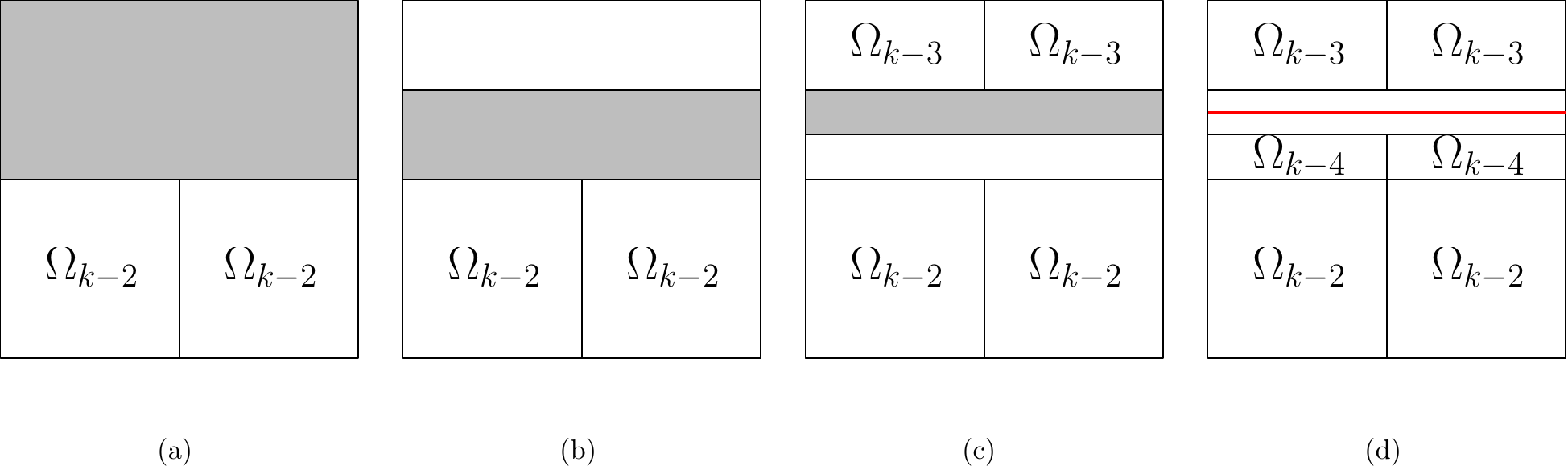}
      \end{center}\vspace{-.5cm}
      \caption{The construction of a tiling in $\partial \Omega_k^|$. The grey areas represent the part that contains the pivotal edge.}
      \label{fig:sum}
   \end{figure}
   
   To estimate the cardinality of $\partial \Omega_k^|$, 
   note that in each step of the construction 
   we have two choices for where the pivotal edge is: either in the top half or the bottom half of the corresponding region. 
   Therefore, the number of tilings in $\partial \Omega_k^|$ is
   $$
     |\partial\Omega_k^|| = \prod_{i=2}^k \left(2|\Omega_{k-i}|^2\right)
     = 2^{k-1} \prod_{i=0}^{k-2} |\Omega_{i}|^2
     = \frac{n}{2} \prod_{i=0}^{k-2} |\Omega_{i}|^2.
   $$
   Hence,
   $$
     \mcE(f)
     = \frac{1}{4 |\Omega_k|} \prod_{i=0}^{k-2} |\Omega_{i}|^2
     \leq \frac{1}{4} \phi^{-2k+2} 
   $$
   where the last step follows from Lemma~\ref{lem:upsilon}.
   Therefore, there exists a constant $c>0$ such that
   $$
      \gamma_k \leq c \phi^{-2k}. 
   $$
   This implies that the relaxation time and mixing time satisfy
   $$    
\trel, \tmix \geq \frac{1}{c} \phi^{2k}  = \frac{1}{c} \phi^{2\log n} = \frac{1}{c} n^{2\log \phi} = \Omega(n^{2\log\phi}).$$ 
This complete the proof of the theorem.
\end{proof}

\section{The spectral gap of the block dynamics}
\label{sec:block-constant}
We now present the proof of Theorem~\ref{thm:block17}, which states that there exists a positive integer $k_0$ such that for all $k \geq k_0$, the spectral gap $\gamma_{k,block}$ is at least $1/17$. 

\begin{proof}[Proof of Theorem~\ref{thm:block17}]
   We start defining the distance between two dyadic tilings $x,y\in\Omega_k$. 
   In order to do this, we recall the notion of \emph{half-bisectors}.
   We say that a tiling $x$ has a \emph{left half-bisector} if the line segment from $(0,1/2)$ to $(1/2,1/2)$ does not intersect the interior of any dyadic rectangle.
   In an analogous way we can define a \emph{right half-bisector} using the line segment from $(1/2,1/2)$ to $(1,1/2)$, 
   a \emph{top half-bisector} using the line segment from $(1/2,1)$ to $(1/2,1/2)$, 
   and a \emph{bottom half-bisector} using the line segment from $(1/2,1/2)$ to $(1/2,0)$.
   Note that if $x$ has a horizontal bisector, then it has both a left half-bisector and a right half-bisector. 
   However, $x$ may have a left half-bisector but no horizontal bisector.
   For example, the dyadic tiling in Figure~\ref{fig:dyadic_ex}(a) has top, right and bottom half-bisectors, but no left half-bisector.
   
   Now we define the distance between $x$ and $y$ as follows. 
   For each of the four possible half-bisectors, let $\ell_1$ be the number of such half-bisectors that are present in either $x$ or $y$, but not in both of them.
   Also, for each of the four possible quadrants (top-left, top-right, bottom-left and bottom-right) of $x$ and $y$, 
   let $\ell_2$ denote the number of such quadrants for which the rectangles in $x$ intersecting that quadrant are not the same as the 
   rectangles in $y$ intersecting that quadrant. Then, introducing a parameter $b>0$ that we will take to be sufficiently large later, we define the distance between $x$ and $y$ as 
   $$
      d(x,y) = b \ell_1 + \ell_2.
   $$
   For instance, consider the two dyadic tilings in Figure~\ref{fig:dyadic_ex}(a,b). In this case we have $\ell_1=1$ due to the left half-bisector that is present in (b) but not in (a), and 
   $\ell_2=3$ for top-left, top-right and bottom-left quadrants. The distance between these two tilings is then $b+3$.
   
   Our goal is to couple two instances of the block dynamics $\Mnb$, one starting from a state $x\in\Omega_k$ and the other from a state $y\in\Omega_k$, 
   such that the distance between $x$ and $y$ contracts after one step of the chains. 
   More precisely, letting $\mathbb{E}_{x,y}$ denote the expectation with respected to the coupling, and if $x'$ and $y'$ are the dyadic tilings obtained after one step of each chain, respectively, 
   we want to obtain a coupling and a value 
   $\Delta>0$ such that
   \begin{equation}
      \mathbb{E}_{x,y}[d(x',y')] \leq (1-\Delta) d(x,y) \quad \text{for all $x,y\in\Omega_k$}.
      \label{eq:chen}
   \end{equation}
   Once we have the above inequality, then a result of Chen
   \cite{Chen} (see also \cite[Theorem~13.1]{lpw}), 
   implies that $\gamma_{k,block}\geq \Delta$.
   
   We will use the following simple coupling between $x'$ and $y'$: 
   \begin{itemize}
      \item Uniformly at random choose a tiling $\rho\in\Omega_{k-1}$.
      \item Uniformly at random choose $Left$, $Right$, $Top$ or $Bottom$.
      \item Retile the choosen half (left, right, top or bottom) of $x$ with $\rho$, if possible.
      \item Retile the choosen half (left, right, top or bottom) of $y$ with $\rho$, if possible.
   \end{itemize}
   For a more detailed description of the retiling step, see the definition of the transition rule of $\Mnb$ in Section~\ref{sec:block}.
   When we update the left (resp., right) half of $x$ and $\rho$ contains a horizontal bisector, note that $x'$ will contain a left (resp., right) half-bisector. 
   Similarly, if we update the top (resp., bottom) half of $x$ and $\rho$ contains a vertical bisector, 
   then $x'$ will contain a top (resp., bottom) half-bisector. 
   In any of these cases, we say that the retiling yields a half-bisector of $x$. 

   The remaining of the proof is devoted to showing that 
   we can set $b$ large enough so that~\eqref{eq:chen} holds with $\Delta=\frac{1}{17}$. 
   In order to see this, we will split into three cases, and show that~\eqref{eq:chen} holds with $\Delta=\frac{1}{17}$
   for each case.

\bigskip\noindent
   \emph{Case 1: $x$ and $y$ have no common bisector}.\\
   The maximum number of common half-bisectors of $x$ and $y$ in this case is two. 
   Figure~\ref{fig:block0} illustrates the three possible configurations for the number of common half-bisectors of $x$ and $y$.
   \begin{figure}
      \begin{center}
         \includegraphics[scale=.9]{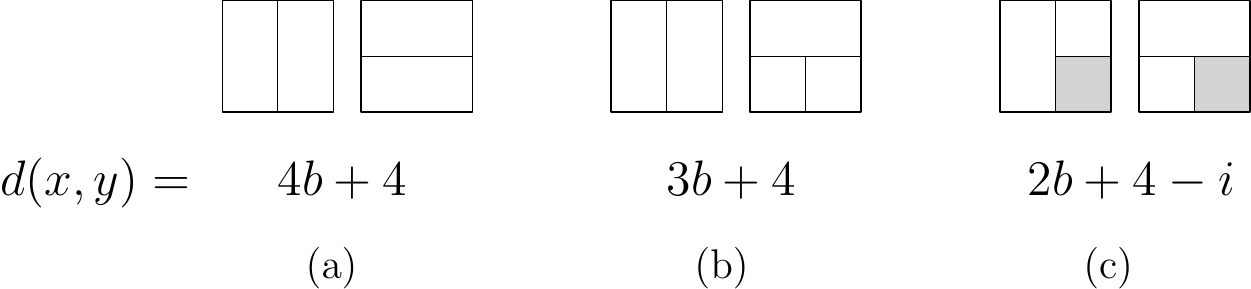}
      \end{center}\vspace{-.5cm}
      \caption{Possible configurations for the half-bisectors of $x$ and $y$ in case 1. 
         In figure (c), $i\in\{0,1\}$ denotes how many grey quadrants are tiled identically in $x$ and $y$.}
      \label{fig:block0}
   \end{figure}
   Consider first that $x$ and $y$ have no common half-bisector, which is illustrated in Figure~\ref{fig:block0}(a) and has $d(x,y)=4b+4$.
   Then, whichever half (left, right, top or bottom) is chosen to be retiled, note that either $x$ or $y$ is actually retiled, but never both. 
   With probability $\frac{|\Omega_{k-2}^2|}{|\Omega_{k-1}|}=f_k$ the retiling yields a half-bisector, which increases the number of 
   common half-bisectors between $x$ and $y$, and thus decreases their distance by $b$. 
   Hence, using that $f_k\geq 1/2$, we have
   $$
      \mathbb{E}_{x,y}[d(x',y')] = d(x,y) - f_k b \leq 4b+4 - \frac{b}{2} < \left(1-\frac{1}{17}\right)(4b+4),
   $$
   where the last step is true by setting $b$ large enough (in this case, $b \geq 1$ suffices).
   
   Now consider that $x$ and $y$ have one common half-bisector, and use Figure~\ref{fig:block0}(b) as a reference, with $x$ being the left tiling and $y$ being the right tiling. We have $d(x,y)=3b+4$.
   If we retile the left or right halves, so only $x$ gets retiled, and the retiling yields a half-bisector, then the number of common half-bisectors of $x$ and $y$ decreases by 1. A similar behavior happens 
   if we retile the top half. However, if we retile the bottom half, and the retiling does not yield a half-bisector, then the number of common half-bisectors decreases by $1$. Hence,
   using that $f_k\geq 1/2$, we obtain
   $$
      \mathbb{E}_{x,y}[d(x',y')] 
      \leq d(x,y) - \frac{3f_k b}{4} + \frac{(1-f_k) b}{4} 
      \leq 3b+4-\frac{b}{4}
      < \left(1-\frac{1}{17}\right)(3b+4),
   $$
   where the last step is true by setting $b$ large enough (in this case, $b \geq 4$ suffices).
   
   Finally, suppose $x$ and $y$ have two common half-bisectors, as illustrated in Figure~\ref{fig:block0}(c), where they may or may not be tiled the same in the quadrant bounded by these common half-bisectors. In this case $d(x,y) = 2b+4-i$, where $i = 1$ if they agree on this quadrant and $i=0$ otherwise. Retiling the left and top halves can yield a new common half-bisector, while 
   retiling the right and bottom halves may remove a common half-bisector. Moreover, if $i=1$ and we retile the right or bottom halves,  the tilings of the bottom-right quadrant of $x$ and of $y$ may become different,
   increasing the distance between $x$ and $y$ by $1$. Putting these together, we have
   \begin{align*}
      \mathbb{E}_{x,y}[d(x',y')] 
      &\leq d(x,y) - \frac{2f_k b}{4} + \frac{2(1-f_k) b}{4} + i\frac{2}{4}
      \\&\leq 2b+4-\frac{i}{2}-\frac{(2f_k-1)b}{2}
      = \frac{(5-2f_k)b}{2}+4-\frac{i}{2}.
   \end{align*}
   Since $f_k \to \frac{\sqrt{5}-1}{2}$ as $k\to\infty$, the right-hand side above goes to $\left(\frac{6-\sqrt{5}}{2}\right)b+4-\frac{i}{2}$. In particular, for $k \geq 10$, the coefficient of $b$  above satisfies $\frac{5-2f_k}{2} < 2\left(1-\frac{1}{17}\right)$, and so we can set $b$ large enough so that 
   $\mathbb{E}_{x,y}[d(x',y')]\leq \left(1-\frac{1}{17}\right)(2b+4-i)$. We note this is the tight case, as $\frac{6-\sqrt{5}}{2} > 2 \left(1-\frac{1}{16}\right)$, so this particular coupling and distance metric cannot be used to show the spectral gap is at least $1/16$. 
   This concludes the first case.

   \bigskip\noindent
   \emph{Case 2: $x$ and $y$ have a common bisector, but neither $x$ nor $y$ has both bisectors}.\\
  Without loss of generality we assume $x$ and $y$ both have a vertical bisector and neither has a horizontal bisector.  Each of $x$ and $y$ has at least $2$ and at most $3$ half-bisectors.
   Figure~\ref{fig:block1} illustrates the four possible configurations for the number of half-bisectors of $x$ and $y$; the shaded quadrants are those where $x$ and $y$ could have the same tiling. 
   \begin{figure}
      \begin{center}
         \includegraphics[scale=.9]{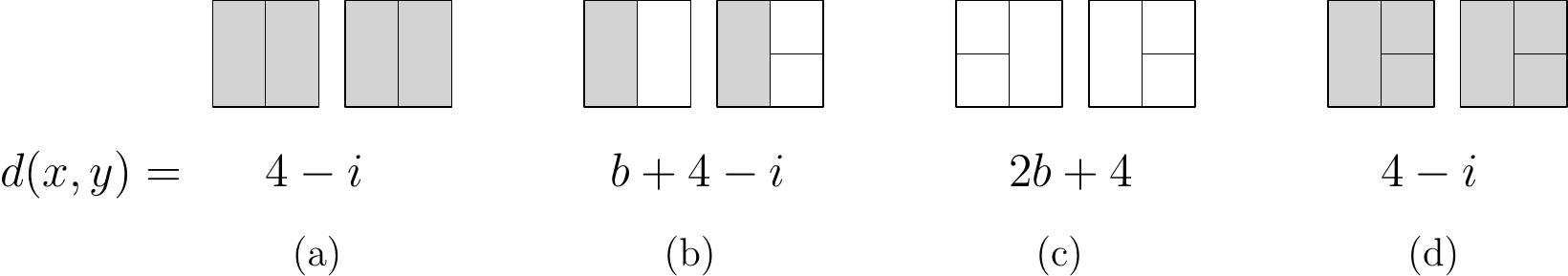}
      \end{center}\vspace{-.5cm}
      \caption{Possible configurations for the half-bisectors of $x$ and $y$ in case 2. 
         The value of $i\in\{0,1,2,3\}$ denotes the number of grey quadrants which is tiled identically in $x$ and $y$.}
      \label{fig:block1}
   \end{figure}
   In all the situations of Figure~\ref{fig:block1}, if we retile the left or right halves, then we match up the configuration of $x$ and $y$ in that half. In particular, if $x$ and $y$ don't agree on the presence of left half-bisector, then they also do not have the same tiling of the top left or bottom left quadrants, so the decrease in distance due to a retiling of the left half, a move that occurs with probability $1/4$, is $(b+2)$. If $x$ and $y$ agree on the presence of a left half-bisector and have the same tiling on $i' \in \{0,1,2\}$ of the two left quadrants, then the decrease in distance due to a retiling of the left half is $(2-i')$. The same holds for right half-bisectors and retilings of the right half. 
   As there are no moves of the coupling that can increase the distance between $x$ and $y$, it can be shown that in all of the cases shown in Figure~\ref{fig:block1} the distance decreases by 1/4 in expectation. Hence, 
   $$
      \mathbb{E}_{x,y}[d(x',y')] 
      \leq d(x,y) - \frac{d(x,y)}{4}
      \leq \left(1-\frac{1}{17}\right)d(x,y),
   $$
   which concludes the second case.

   \bigskip\noindent
   \emph{Case 3: $y$ has both vertical and horizontal bisectors}.\\
   Here there are three situations, depending on whether $x$ has two, three or four half-bisectors; see Figure~\ref{fig:block2}.
   \begin{figure}
      \begin{center}
         \includegraphics[scale=.9]{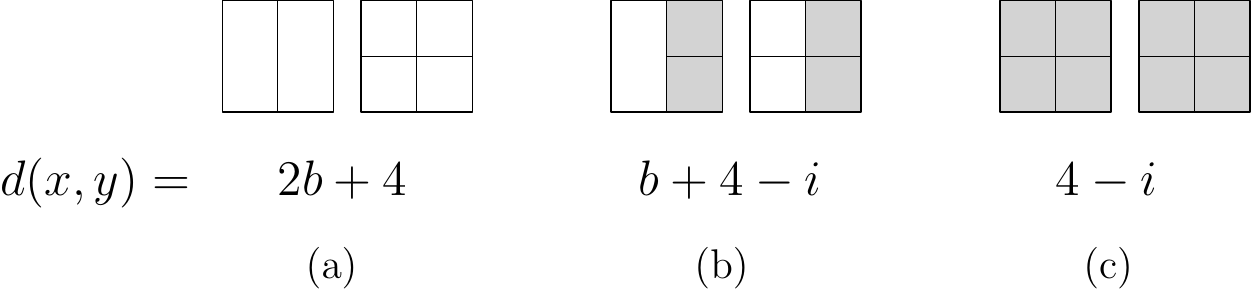}
      \end{center}\vspace{-.5cm}
      \caption{Possible configurations for the half-bisectors of $x$ and $y$ in case 3. 
         The value of $i\in\{0,1,2,3\}$ denotes the number of grey quadrants which is tiled identically in $x$ and $y$.}
      \label{fig:block2}
   \end{figure}
   In the situation of Figure~\ref{fig:block2}(a), if the left or right halves are retiled, then we match up $x$ and $y$ in that half, decreasing the distance by $b+2$. But if we retile the top or bottom halves, then we 
   may increase the distance by $b$ if the retiling does not yield a half-bisector. Hence,
   $$
      \mathbb{E}_{x,y}[d(x',y')] 
      \leq d(x,y) - \frac{2(b+2)}{4} + \frac{2(1-f_k) b}{4}
      =\frac{(4-f_k)b}{2} +3.
   $$
   Since $\frac{4-f_k}{2}\to \frac{9-\sqrt{5}}{4}< \left(1-\frac{1}{17}\right)2$, the right-hand side above is smaller than $\left(1-\frac{1}{17}\right)(2b+4)$ when $k$ and $b$ are large enough.
   A similar situation occurs in Figure~\ref{fig:block2}(b), but the distance increases a bit more when the top or bottom half is retiled as quadrants that were equal in $x$ and $y$ may become different. In this case,
   we have
   $$
      \mathbb{E}_{x,y}[d(x',y')] 
      \leq d(x,y) - \frac{(b+4-i)}{4} + \frac{2(1-f_k) b}{4} + \frac{2}{4}
      =\frac{(5-2f_k)b}{4} +\frac{6-i}{4}.
   $$
   Since $\frac{5-2f_k}{4}\to \frac{6-\sqrt{5}}{4}<\left(1-\frac{1}{17}\right)$, the right-hand side above is smaller than $\left(1-\frac{1}{17}\right)(b+4-i)$ when $k$ and $b$ are large enough; this is the second tight case, where we see contraction by a factor of $1-\frac{1}{17}$ but not by $1-\frac{1}{16}$.
   Finally, for the situation in Figure~\ref{fig:block2}(c), regardless of which half we choose to retile, the distance will not increase; if we choose a half containing a quadrant on which $x$ and $y$ differ, the distance will decrease. Each quadrant on which $x$ and $y$ differ is contained in two halves and thus is retiled so that $x$ and $y$ agree there with probability $1/2$. 
   That is,
   $$
      \mathbb{E}_{x,y}[d(x',y')] 
      \leq d(x,y) - \frac{ d(x,y)}{2}
      \leq \left(1-\frac{1}{17}\right)d(x,y).
   $$
   This concludes the third case. We have shown that for all possible tilings $x$ and $y$, it holds that ${\mathbb{E}_{x,y}[d(x',y')] \leq \left(1-\frac{1}{17}\right)d(x,y)}$. This implies $\gamma_{k,block} \geq \frac{1}{17}$ for all $k$ sufficiently large, as desired. 
\end{proof}





\section*{Acknowledgments}
This work started during the 2016 AIM workshop \emph{Markov chain mixing times}. 
We thank the organizers for the invitation and the stimulating atmosphere. 

\bibliographystyle{plain}

\bibliography{biblio}

\end{document}